\documentclass[12pt]{amsart}

\usepackage[utf8]{inputenc}
\usepackage[english]{babel}
 
\usepackage{geometry}
\usepackage{anysize}
\usepackage{amsfonts}
\usepackage{amssymb}
\usepackage{amsmath}
\usepackage{amsthm}
\usepackage{amsrefs} 
\usepackage[all]{xy}
\usepackage{enumerate}
\usepackage{multirow}
\usepackage{color}
\usepackage{colonequals}
\usepackage{mathrsfs}

\definecolor{darkgreen}{rgb}{0,0.5,0}

\usepackage[
        draft=false,
        colorlinks, citecolor=darkgreen,
        pdfauthor={Filippo F. Favale and Sonia Brivio},
        pdftitle={Kernel bundles over reducible curves with a node},
        linktocpage        
]{hyperref}

\newtheorem{definition}{Definition}[section]

\newtheorem{proposition}{Proposition}[section]
\newtheorem{corollary}[proposition]{Corollary}
\newtheorem*{Conjecture*}{Conjecture}
\newtheorem*{theorem*}{Theorem}
\newtheorem{lemma}[proposition]{Lemma}
\newtheorem{theorem}[proposition]{Theorem}

\newtheorem{remark}{Remark}[proposition]

\DeclareMathOperator{\GL}{GL}
\DeclareMathOperator{\Ker}{Ker}

\DeclareMathOperator{\Pic}{Pic}

\DeclareMathOperator{\Tor}{Tor}

\DeclareMathOperator{\Rk}{Rk}

\newcommand{\OO}{\mathcal{O}}

\newcommand{\ev}{ev}

\numberwithin{equation}{section}

\begin{document}

\title[Kernel bundles on nodal curves]{On kernel bundles over reducible curves with a node}

\author{Sonia Brivio}
\address{Dipartimento di Matematica e Applicazioni,
	Universit\`a degli Studi di Milano-Bicocca,
	Via Roberto Cozzi, 55,
	I-20125 Milano, Italy}
\email{sonia.brivio@unimib.it}

\author{Filippo F. Favale}
\address{Dipartimento di Matematica e Applicazioni,
	Universit\`a degli Studi di Milano-Bicocca,
	Via Roberto Cozzi, 55,
	I-20125 Milano, Italy}
\email{filippo.favale@unimib.it}

\date{\today}
\thanks{
Both authors are members of GNSAGA-INdAM}

\maketitle
\pagestyle{plain}

\begin{abstract}
Given a vector bundle $E$ on a complex reduced curve $C$ and a subspace $V$ of $H^0(E)$ which generates $E$, one can consider the kernel of the evaluation map $ev_V:V\otimes \OO_C\to E$, i.e. the {\it kernel bundle } $M_{E,V}$ associated to the pair $(E,V)$. Motivated by a well known conjecture of Butler about the semistability of $M_{E,V}$ and by the results obtained by several authors when the ambient space is a smooth curve, we investigate the case of a reducible curve with one node. Unexpectedly, we are able to prove results which goes in the opposite direction with respect to what is known in the smooth case. For example, $M_{E,H^0(E)}$ is actually quite never $w$-semistable. Conditions which gives the $w$-semistability of $M_{E,V}$ when $V\subset H^0(E)$ or when $E$ is a line bundle are then given.
\end{abstract}

\section*{Introduction}
Let $C$ be a complex reduced projective  curve.  Let $(E,V)$ be  a  pair on $C$  given by  a  vector  bundle $E$  of rank $r$  and a vector space $V \subseteq H^0(E)$ of dimension $k \geq r+1$. These pairs have been studied extensively by many authors, in particular when $C$ is smooth, and are called {\it coherent systems}. For example, one can see \cite{BGMN} and \cite{KN}.
In this paper we will be interested in pairs  where  $E$ is generated by the global sections of $V$, i.e. when the evaluation map  
$\ev_V:V\otimes \OO_C\to E$ is surjective: we will call these pairs  {\it generated pairs}. In these case, the kernel $M_{E,V}$ of the evaluation map  $\ev_V$ is a vector bundle of rank $k-r$ on $C$ which fits into the following exact sequence
\begin{equation}
\label{DSB}
\xymatrix{
0\ar[r] & M_{E,V}\ar[r] & V\otimes \OO_C\ar[r]^-{\ev_V} & E\ar[r] & 0
}
\end{equation}
and it is called the {\it kernel bundle} (or the {\it Lazarfeld bundle})  of the pair $(E,V)$.  When $V = H^0(E)$, then it is denoted by $M_E$ and this case is said the {\it complete case}.  
\hfill\par
Generated pairs encode a lot of the geometry of the  curve as well as a lot of interesting informations about it. For example,  a generated pair $(E,V)$  defines a morphism  
$$\varphi_{E,V} \colon C \to G(k-r,V), \quad x \to \Ker(ev_{V,x}),$$
where $G(k-r,V)$ denotes the Grassmannian variety of $(k-r)$-dimensional linear  subspaces of $V$. Then, the exact sequence \ref{DSB} is actually the pull back by $\varphi_{E,V}$ of the following exact sequence on $G(k-r,V)$:
\begin{equation}
\xymatrix{
0 \ar[r]& 
{\mathcal U} \ar[r]&
V \otimes \OO_{G(k-r,V)} \ar[r]&
{\mathcal Q} \ar[r] &
0}
\end{equation}
where ${\mathcal U}$ and ${\mathcal Q}$ are respectively the  universal and quotient bundle on $G(k-r,V)$. 
\hfill\par
Kernel  bundles have been studied by many authors when the curve $C$ is smooth and irreducible, because of their rich applications (see \cite{La} for an overview). In particular, their stability properties, which have been studied with respect to different point of view (see, for instance, \cite{EL,M}) are closely related to higher rank Brill-Noether theory and moduli spaces of coherent systems (see \cite{BGMN,BBN1,BB} for example). Finally, they have been useful  in studying  theta divisors and the geometry of moduli space of vector bundles on curves (see  \cite{Be, B15,B17, Br,BF1,BV,Po}, for example). 
\\ \hfill \par

\noindent In the complete case, Butler, in his seminal work \cite{B1}, proved that on a smooth irreducible  curve of  genus $g$ the kernel bundle  $M_E$ is semistable  for any semistable $E$ of degree $d \geq 2rg$ and it is actually stable if $E$ is stable of degree $d>2gr$. 
In the case of line bundles, this result has been improved in several works by taking into consideration the Clifford index of the curve. For example, see \cite{B2,B-P,C,MS}.
\\ \hfill\par

For the general case, in \cite{B2} Butler made the following conjecture:

\begin{Conjecture*}
For a general smooth curve of genus $g \geq 3$  and a general choice of a generated pair $(E,V)$, where $E$ is  a semistable vector bundle, the kernel bundle $M_{E,V}$ is semistable. 
\end{Conjecture*}

Much work has been done in the direction of solving this conjecture. In particular it has been completely proved in \cite{BBN2} in the case of line bundles. Moreover, many conditions for stability are given (see also \cite{BBN1}). 
\\ \hfill\par

It seems natural to ask whether similar results hold in the case of singular curves. The aim of this paper is to investigate stability properties for kernel bundles on a nodal reducible curve. More precisely, we will consider a complex reducible projective  curve $C$ with two smooth irreducible components $C_i$, of genus $g_i \geq 2$, $i = 1,2$, and a single node $p$. 
Some modifications to the environment are required, obviously. For example, the notion of semistability needs to be replaced with the notion of $w$-semistability for a given polarization $w$ (see Definition \ref{DEF:wSTAB}).  We will say that a  vector bundle on $C$ is {\it strongly unstable} if it is $w$-unstable for any polarization on the curve (see Definition \ref{DEF:STRUNST}). The theory and the results that we will need are summarized in Section \ref{sec1}.
\\ \hfill\par

Our first result is Theorem \ref{MEVunstable}, which is proved in Section \ref{sec2}. 
\begin{theorem*}
Let $C$ be a nodal curve as above. Let $(E,V)$ be  a generated pair on the curve $C$   and let $E_i$ be the restriction of $E$ to the component $C_i$.
If $E_i$ is semistable and $\dim (V \cap H^0(E_i(-p)))\geq 1$, then $M_{E,V}$ is strongly unstable and both its restrictions to the components are unstable. 
\end{theorem*}

As a corollary of this result, we have that the kernel bundle $M_E$ is strongly unstable for any globally generated vector bundle $E$ whose restrictions $E_i$ are semistable and not trivial (see Corollary \ref{cor2}). This is a somewhat unexpected result as in the smooth case, by the result of Butler, $M_E$ is always semistable for any semistable $E$ with degree sufficiently big. This impressive difference shows that it is worth going  on  this kind of problems. 
\\ \hfill \par

It is natural to consider the restrictions of a generated pair $(E,V)$ to the component $C_i$. In this way, we show that we get again a generated pair 
$(E_i,V_i)$. Nevertheless, it is not always true that the kernel bundle associated to $(E_i,V_i)$ is the restriction of $M_{E,V}$ to $C_i$. Indeed in the case of Theorem \ref{MEVunstable}, we have that $M_{E_i,V_i}$ is a destabilizing quotient for the restriction of $M_{E,V}$ to $C_i$.
\\ \hfill\par

In section \ref{sec3} we study the stability of $M_{E,V}$ when ${M_{E,V}}_{\vert_{C_i}} \simeq M_{E_i,V_i}$. We first show that this is equivalent to the following condition: 
$$
(\star) \qquad \qquad \qquad \qquad V \cap H^0(E_1(-p))= V \cap H^0(E_2(-p)) = \{ 0 \}.\quad\qquad  \qquad \qquad \qquad \qquad$$
Then, we prove Theorem \ref{THM:wstab}:

\begin{theorem*}
Let $C$ be  a nodal curve as above. 
Let $(E,V)$ be a generated pair  on the  curve $C$  satisfying condition $(\star)$. If both $M_{E_1,V_1}$ and $M_{E_2,V_2}$ are semistable then there exists a polarization $w$ such that $M_{E,V}$ is $w$-semistable.
\end{theorem*}

We conclude Section \ref{sec3} by applying the above theorem to special cases where we know semistability of the  kernel bundles $M_{E_i,V_i}$. The results are stated in Theorem \ref{THM:MAIN2} and \ref{THM:3}. 
\\ \hfill\par

{\bf Acknowledgements}. The authors want to express their gratitude to the anonymous referee for helpful remarks and to Prof. P.E. Newstead for his keen suggestions which contributed to the final version of this paper. 

\section{Vector bundles on nodal reducible curves}
\label{sec1}
 Let $C$ be a complex projective curve with two smooth irreducible  components and one single node   $p$, i.e. $p$ is an ordinary double point. 
Assume that  
$$ \nu \colon C_1 \sqcup C_2 \to C$$ is a  normalization map for $C$, where $C_i$ is a smooth irreducible projective curve of genus $g_i \geq 2$.  
Then ${\nu}^{-1}(x)$ is a single point except when $x$ is the node,   in the latter  case  we have: ${\nu}^{-1}(p) = \{ q_{1}, q_{2} \}$ with $q_{i} \in C_i$.
Since for $i = 1,2$ the restriction 
$$\nu_{\vert C_i} \colon C_i \to \nu(C_i)$$ is an isomorphism, we will call as well $C_1$ and $C_2$ the components of $C$.
As $C$ is of compact type, i.e. every node of $C$ disconnects
the curve, the pull back map $\nu^*$ induces an isomorphism 
$$\Pic (C) \simeq \Pic(C_1 \sqcup C_2)\simeq \Pic(C_1)\times  \Pic(C_2).$$
Moreover, the curve $C$ is contained in a smooth irreducible projective surface $X$. On it, $C$, $C_1$ and $C_2$ are effective divisors and we have: $C = C_1 + C_2$  with $C_1\cdot C_2 = 1$.   In particular $C$ is $1$-connected, so  the sheaf  $\omega_C=\omega_X\otimes \OO_C(C)$ is a  dualizing sheaf  for $C$ and   we have $h^1(C,\omega_C)= 1$, see \cite{FT}. 

Moreover, for $i,j\in\{1,2\}$ with $i\neq j$, we have an exact sequence
\begin{equation}
\label{decomposition}
\xymatrix{
0\ar[r] &
\OO_{C_i}(-C_j)\ar[r] &
\OO_C\ar[r] & 
\OO_{C_j}\ar[r] & 0
}
\end{equation}
from which we get
$$\chi(\OO_C)= \chi(\OO_{C_i}(-C_j)) + \chi(\OO_{C_j})=\chi(\OO_{C_i}(-p)) + \chi(\OO_{C_j}).$$
Let $p_a(C) = 1 - \chi(\OO_C)$ be the arithmetic genus of $C$, then we have:
$$p_a(C) = g_1+g_2.$$
Let  $\overline{\mathcal M}_g$ be the moduli space of stable curves of arithmetic genus $g$, it is a projective integral scheme containing as a dense subset the moduli space 
${\mathcal M}_g$ of smooth curves of genus $g$. The closure in $\overline{\mathcal M}_g$ of the locus of curves with a single node has codimension $1$ and is the union of finitely many irreducible divisors $\Delta_i$, where $\Delta_0$ parametrizes irreducible nodal curves whereas, for $0<i\leq\left[g/2\right]$, $\Delta_i$ parametrizes reducible nodal curves with two smooth irreducible components of genus $i$ and $g-i$, see \cite{DM} and \cite{Gie}. Let $C$ be a nodal curve as above with $p_a(C)= g_1 + g_2$ and assume $g_1\leq g_2$. We will say that $C$ is a general reducible nodal curve if it is a general element of the divisor $\Delta_{g_1}$. In this case, both of its components are general in their moduli spaces.
\\ \hfill\par

To talk about semistability for vector bundles on  a reducible curve, we first need to introduce the notions of depth  one sheaves following \cite{S}. 

\begin{definition}
A coherent sheaf $E$ on $C$ is of depth one if for any point $x \in C$ the stalk $E_x$ is a $\OO_{C,x}$-module of depth one, i.e. there exists an element in the maximal ideal  ${m}_{x}$ of $ \OO_{C,x}$ which is not a zero divisor for $E_x$. 
\end{definition}
Any vector bundle  on $C$ is a sheaf of depth one,  any subsheaf of a vector bundle too.  If $E$ is a sheaf of depth one, then its restriction to $C_i$ is a torsion free sheaf on $C_i \setminus p$. 
In particular, let $j_i \colon C_i \hookrightarrow C$ be the natural inclusion, if $F_i$ is a vector bundle on  $C_i$, then the sheaf $(j_i)_*(F_i)$ is a depth one sheaf on $C$ and 
$\chi(C_i,F_i) = \chi(C,(j_i)_*(F_i))$. 
In the sequel, when no confusion arises we will denote with 
$F_i$ on $C$ the sheaf $(j_i)_*(F_i)$.
\hfill\par
Let $E$ be a depth one sheaf on $C$, let $E_i$ denote the restriction on the component $C_i$ modulo torsion. We define the {\it relative rank }  and the {\it relative degree} of $E$ with respect to the component $C_i$ as follows:
$$r_i = \Rk (E_i),  \quad d_i = deg(E_i) = \chi(E_i) - r_i \chi(O_{C_i}),$$
where $\chi(E_i)$ is the Euler characteristic of $E_i$. We say that $E$ has  {\it multirank} $(r_1,r_2)$ and {\it multidegree} $(d_1,d_2)$. 
For a vector bundle $E$, the restriction $E_i$ is a  vector bundles too and  we have $r_1= r_2= r$,  we denote $r$ as the {\it rank }  of $E$. 

\begin{definition}
A polarization $w$ of $C$ is given  by a pair of rational weights $(w_1,w_2)$ such that $0 < w_i <1$ and $w_1 + w_2 = 1$. For any sheaf $E$ of depth one on $C$ of multirank $(r_1,r_2)$ we define the polarized slope as 
$$ \mu_{w}(E) = \frac{\chi(E)}{w_1r_1 + w_2r_2}.$$
\end{definition}

\begin{definition} 
\label{DEF:wSTAB}
A vector bundle $E$ on $C$ is called $w$-semistable if for any proper subsheaf $F \subset E$ we have $\mu_w(F) \leq \mu_w(E)$; it is said to be $w$-stable if $\mu_w(F) < \mu_w(E)$.
\end{definition}

If $w$ and $w'$ are two polarizations, it can happen that a depth one sheaf $E$ is $w$-semistable and it is not $w'$-semistable. 

\begin{definition}
\label{DEF:STRUNST}
Let $C$ be as above. We will say that a vector bundle $E$ on $C$ is {\it strongly unstable} if it is $w$-unstable with respect to any polarization on $C$.
\end{definition}

We will need the following result, see \cites{T1, T2}. 
\begin{theorem}
\label{stabilityconditions}
Let $C$ be a nodal curve with two smooth irreducible components $C_i$ and a single node. Let $E$ be a vector bundle on $C$ of rank $r \geq 1$ with restrictions $E_i$, $i = 1,2$ and $w$ a polarization. 
Then we have the following properties:
\begin{enumerate}
\item{} if $E$ is $w-$semistable then:
\begin{equation}
\label{COND:TEIXI}
w_i \chi(E) \leq \chi(E_i) \leq w_i \chi(E) + r;
\end{equation}
\item{} if $E_1$  and $E_2$ are  semistable and $E$ satisfies the above condition, then $E$ is $w$-semistable. Moreover, if at least one of the restrictions is stable, then $E$ is $w$-stable too.
\end{enumerate}
\end{theorem}

\section{Strongly unstable kernel bundles}
\label{sec2}
Let $C$ be a nodal reducible curve with two smooth irreducible components $C_i$ of genus $g_i \geq 2$ and  a single node  $p$ as in Section \ref{sec1}.
Let $(E,V)$ be a generated pair  on $C$, with $E$  a vector bundle of rank $r\geq 1$ on $C$ and $V \subseteq H^0(E)$ of dimension $k \geq r+1$.  Consider the kernel bundle $M_{E,V}$ associated to the pair $(E, V)$. It is a vector bundle of rank $k -r \geq 1$ and Euler-characteristic  
$$\chi (M_{E,V}) =  k (1- p_a(C)) - \chi(E).$$ 
Let $E_i$ be the restriction of $E$ to the component $C_i$ and let $d_i$ be the degree of $E_i$.
Let 
\begin{equation}
\rho_i \colon H^0(E) \to H^0(E_i)
\end{equation} 
be the restriction map of global sections of $E$ to the component $C_i$ and let's consider its image
$$V_i = \rho_i(V).$$
We have a pair $(E_i,V_i)$ on the curve $C_i$. This pair is said to be of type
$(r,d_i,k_i)$ as $r=\Rk(E_i)$, $d_i$ is the degree of $E_i$ and $k_i = \dim V_i$. We will denote $(E_i,V_i)$ as the restriction of the pair $(E,V)$ to the curve $C_i$. 
\\ \hfill\par

We have the following lemmas:

\begin{lemma}
\label{lemmatecnico0}
Let $C$ be  a nodal reducible curve as in Section \ref{sec1}.
Let $(E,V)$ be a  generated pair on $C$. Then
\begin{enumerate}
\item $(E_i,V_i)$ is a generated pair on the curve $C_i$ and $k_i\geq r$, $i = 1,2$;
\item $h^0(E) = h^0(E_1) + h^0(E_2) -r$;
\item{} the restriction map $\rho_i \colon H^0(E) \to H^0(E_i)$ is surjective, for $i = 1,2$. 
\end{enumerate}
\end{lemma}
\begin{proof}
(1) Consider the exact sequence \ref{DSB} defining $M_{E,V}$. Since $\Tor^1_{\OO_C}(\OO_{C_i},E) = 0$ by tensoring the exact sequence with $\OO_{C_i}$ we get the following exact sequence of locally free sheaves on $C_i$: 
$$
\xymatrix{
 0\ar[r] & M_{E,V} \otimes O_{C_i} \ar[r] & V \otimes \OO_{C_i}\ar[r]^-{{\ev_{V}}_{\vert{ C_i} }}& E_i\ar[r] & 0.
}
$$
We have the following commutative diagram
\begin{equation}
\label{CD1}
\xymatrix{
0 \ar[r] & M_{E,V} \otimes \OO_{C_i} \ar[r]  &
V \otimes \OO_{C_i} \ar[r]^-{{ev_V}_{\vert C_i}} \ar[d]^-{{\rho_i}_{\vert V}
} & E_i \ar[r] \ar@{=}[d] & 0  \\
 &  & V_i \otimes \OO_{C_i} \ar[r]_-{{ev_{V_i}}} & E_i  }
\end{equation}
From it we deduce that $ev_{V_i}$ is surjective and thus, $(E_i,V_i)$ is a generated pair. In particular, $k_i=\dim(V_i)\geq r$.
\\ \hfill\par

(2) Following \cite{S}, the vector bundle $E$ on the curve $C$ is
obtained by gluing  the vector bundles $E_1$ and $E_2$ along the fibers at the node. More precisely $E$  is determined  by the   triple $(E_1,E_2,\sigma)$, where $\sigma \in \GL(E_{1,p},E_{2,p})$  is an isomorphism.
In particular, $\sigma$ induces  the following  commutative diagram, (see \cite{BF2} for more details)
\begin{equation}
\label{EQ:OLDPAPER}
\xymatrix{
0 \ar[r] &  E    \ar[r] &  
E_1 \oplus  E_2   \ar[r]  \ar[d]^{{\rho}_{1,p}\oplus{\rho}_{2,p}} & E_{2,p}  \ar[r]  \ar@{=}[d]   & 0  \\
 &  & E_{1,p}  \oplus E_{2,p} \ar[r]_-{\delta} & E_{2,p} \ar[r]  & 0 }
\end{equation}
where $\rho_{i,p}$ is the restriction map to the fiber at $p$ and 
$\delta(u,v) = \sigma(u)-v$. 
Since by (1) $E_i$ is globally generated, passing to cohomology we obtain the exact sequence
$$ 0 \to H^0(E) \to H^0(E_1) \oplus H^0(E_2) \to E_{2,p} \to 0,$$
which proves (2).
\\ \hfill\par

(3) Let $i,j \in \{1,2\}$ with $i\neq j$. If we tensor the exact sequence \ref{decomposition} with $E$, we obtain
$$0 \to E_j(-p) \to E \to E_i \to 0.$$
Passing to cohomology we have:
$$ \xymatrix{
0 \ar[r] & H^0(E_j(-p)) \ar[r] &  H^0(E) \ar[r]^{\rho_i} \ar[r]&  H^0(E_i) \ar[r]  & ...}$$
from which we deduce that  $\Ker \rho_i \simeq H^0(E_j(-p))$. Since $E_j$ is globally generated   we have:
$$\Rk \rho_i = h^0(E) - h^0(E_j(-p)) = 
h^0(E) - h^0(E_j) + r = h^0(E_i).$$
\end{proof}

\begin{lemma}
\label{lemmatecnico1}
Let $C$ be  a nodal reducible curve as in Section \ref{sec1}.
Let $(E,V)$ be a  generated pair on $C$. If $E_i$ is semistable and $\dim (V \cap H^0(E_i(-p))) \geq 1$, $i = 1,2$, then we have:
\begin{enumerate}
\item{} $(E_i,V_i)$ is of type $(r,d_i,k_i)$, with  $d_i  \geq r$ and $k_i \leq k-1$;
\item{} the kernel bundle $M_{E_i,V_i}$ is a non trivial  quotient of the  restriction $M_{E,V} \otimes \OO_{C_i}$;
\item{}  the restriction $M_{E,V} \otimes \OO_{C_i}$ is an unstable vector bundle on $C_i$. 
\end{enumerate}
\end{lemma}
\begin{proof}

(1)-(2) First of all note that the assumption  $ \dim (V \cap H^0(E_i(-p))) \geq 1$ implies $d_i \geq r$.  Indeed, as $ E_i(-p)$ has a non zero global section $s_i$, the zero locus of $s_i$ is an effective divisor $Z_0(s_i)$ on $C_i$ of degree at least $1$ and we have an inclusion  of sheaves $\OO_{C_i}(Z_0(s_i)) \subset E_i$.  Since $E_i$ is semistable,  we  must have 
$$1 \leq \deg(Z_0(s_i)) = \mu(\OO_{C_i}(Z_0(s_i))) \leq \mu (E_i) = d_i/r,$$
which gives $d_i \geq r$. 
\\ \hfill\par
Let $S_j = V \cap H^0(E_j(-p))$.  By Lemma \ref{lemmatecnico0}
and its proof, $\Ker(\rho_i)\simeq  H^0(E_j(-p))$ and $Im \rho_i \simeq H^0(E_i)$. Hence we have the  following exact sequence:
$$ \xymatrix{
0 \ar[r] & S_j \ar[r] &  V \ar[r]^{{\rho_i}_{\vert V}
} \ar[r]&  V_i \ar[r]  & 0,}
$$
where $\dim S_j \geq 1$ and $\dim V_i \leq k-1 $. 
\\ \hfill\par

We can then complete  diagram  \ref{CD1} obtaining the following one:
\begin{equation}
    \label{mainCD}
\xymatrix{
         &  0 \ar[d]  & 0 \ar[d]  &  & & \\
0 \ar[r] &  S_j \otimes \OO_{C_i} \ar@{=}[r] \ar[d] &  S_j \otimes \OO_{C_i} \ar[r] \ar[d] &   0 \ar[d]   & \\
0 \ar[r] & M_{E,V} \otimes \OO_{C_i} \ar[r] \ar[d]^-{{\rho_i}_{\vert V}
} &
V \otimes \OO_{C_i} \ar[r]^-{{ev_V}_{\vert C_i}} \ar[d]^-{{\rho_i}_{\vert V}
} & E_i \ar[r] \ar@{=}[d] & 0  \\
0 \ar[r] & M_{E_i,V_i} \ar[r] \ar[d] & V_i \otimes \OO_{C_i} \ar[r]_-{{ev_{V_i}}} \ar[d] & E_i \ar[r] \ar[d] & 0 \\
 &   0  & 0 & 0 & }
\end{equation}
By definition the kernel of $ev_{V_i}$ is the kernel bundle $M_{E_i,V_i}$, which turns out to be  a non trivial quotient of $M_{E,V} \otimes \OO_{C_i}$.
\\ \hfill\par

(3) We  claim that $S_j \otimes \OO_{C_i}$ is a non trivial destabilizing subsheaf  of $M_{E,V} \otimes \OO_{C_i}$. In fact we have: 
$$0=\mu(  S_j \otimes \OO_{C_i} )  > \mu( M_{E,V} \otimes \OO_{C_i} )= \frac{-d_i}{k-r} <0,$$
since as we have seen $d_i \geq r$. 
\end{proof}

\begin{remark}
Note that, under the assumption of  Lemma  \ref{lemmatecnico1}, $M_{E_i,V_i}$ is actually a destabilizing quotient for $M_{E,V}\otimes \OO_{C_i}$.
\end{remark} 

\begin{remark}
\label{REM:TObeprecise}
Let $(E,V)$ be a generated pair with $E_i$ semistable and not trivial. Then
\begin{enumerate}
\item $h^0(E_i(-p))\geq 1$.
\\ \par

\noindent {\rm Indeed, the restriction $(E_i,V_i)$ of $(E,V)$ is a generated pair and $\dim(V_i)\geq r$ by Lemma \ref{lemmatecnico0}. Equality holds if and only if $E_i\simeq V_i\otimes \OO_{C_i}$ which is impossible by assumption. Hence $h^0(E_i)\geq \dim(V_i)\geq r+1$ and then $h^0(E_i(-p))\geq 1$}. 
\\ \par

\item If, moreover, $\dim(V)>h^0(E_i)$ then $V\cap H^0(E_j(-p))\neq\{ 0\}$.\\ \par

\noindent {\rm It is enough to notice that the restriction $\rho_i|_V:V\to H^0(E_i)$ cannot be injective in this case. Hence $V\cap\Ker(\rho_i)=V\cap H^0(E_j(-p))\neq \{0\}$.}
\end{enumerate}
\end{remark}

In the sequel we will need the following result of Xiao  which generalizes Clifford's Theorem to semistable vector bundles of rank $r \geq 2$ on smooth curves.
\begin{theorem}
\label{THM:XIAOCLIFFORD}
Let $E$ be a semistable vector bundle of rank $r \geq 1$ on a smooth irreducible complex projective curve $C$ of genus $g \geq 2$. If we assume that $0\leq \mu(E)\leq 2g-2$, then
$$h^0(E)\leq \deg(E)/2+r.$$
\end{theorem}
For the proof see \cite{BGN}. 
\\ \hfill\par

Our first result is the following condition for strongly unstable kernel bundles:

\begin{theorem}
\label{MEVunstable}
Let $C$ be a reducible nodal curve as in Section \ref{sec1}. 
Let $(E,V)$ be a generated pair on $C$. 
If $E_i$ is semistable and $ \dim (V \cap H^0(E_i(-p))) \geq 1$, $i = 1,2$, then the kernel bundle $M_{E,V}$ is strongly unstable. 
\end{theorem}
\begin{proof}
We have to prove that $M_{E,V}$ is $w$-unstable with respect to any polarization $w$ on $C$. Assume on the contrary, that there exists a polarization $w= (w_1,w_2)$ such that $M_{E,V}$ is $w$-semistable. As in the proof of Lemma \ref{lemmatecnico1},  let  $S_j = V \cap H^0(E_j(-p))$ and $s_j = \dim S_j \geq 1$, $j = 1,2$.
As we have proved in Lemma \ref{lemmatecnico1}, for $i \not= j$,  the sheaf $S_j \otimes \OO_{C_i}$ is a  subsheaf of $M_{E,V} \otimes \OO_{C_i}$. Hence we also have the following  inclusion of locally free sheaves on $C_i$: 
$$ 
\xymatrix{
S_j \otimes \OO_{C_i}(-p) \ar@{^{(}->}[r]& M_{E,V} \otimes \OO_{C_i}(-p).
}
$$
If we tensor with $M_{E,V}$ the exact sequence \ref{decomposition} we obtain
$$ 0 \to M_{E,V} \otimes \OO_{C_i}(-p) \to M_{E,V} \to M_{E,V} \otimes \OO_{C_j} \to 0,$$
from which we deduce the following inclusion of sheaves  of depth one on $C$:
$$ 
\xymatrix{
S_j \otimes \OO_{C_i}(-p) \ar@{^{(}->}[r]& M_{E,V}.
}
$$
Since  we are assuming that $M_{E,V}$ is $w$-semistable, we have
\begin{equation}
\label{EQ:SEMISTABASS}
\mu_w(S_j \otimes \OO_{C_i}(-p)) \leq \mu_w( M_{E,V}).
\end{equation}
As
$$\mu_w(S_j\otimes \OO_{C_i}(-p)) = \frac{\chi( S_j \otimes \OO_{C_i}(-p))}{w_is_j}= \frac{-g_i}{w_i},$$
and
$$\mu_w(M_{E,V}) = \frac{\chi(M_{E,V})}{k-r}= -\frac{(d_1+d_2) +(k -r)(p_a(C)-1)}{k -r},$$
from the inequality \ref{EQ:SEMISTABASS} we obtain
\begin{equation}
\label{INEQ1_E}
w_i \leq \frac{g_i(k -r)}{d_1 + d_2 + (k -r)(p_a(C)-1)}, \quad i = 1,2. 
\end{equation}
Now recall that $(w_1,w_2)$ is a polarization so $w_1+w_2=1$. Hence, 
\begin{equation}
\label{INEQ_E}
1=w_1+w_2 \leq \frac{p_a(C)(k -r)}{d_1 + d_2 + (k -r)(p_a(C)-1)}
=\frac{p_a(C)(k -r)}{p_a(C)(k-r)+(d_1 +d_2-k+r)}.
\end{equation}
\hfill\par

{\bf Claim}:  $k< d_1 + d_2 +r$.
\\ \hfill\par

First of all, by Lemma \ref{lemmatecnico0} we have: 
\begin{equation}
\label{EQ:INEQH0}
k \leq h^0(E) = h^0(E_1)+h^0(E_2)-r.
\end{equation}
In order to prove the claim, we can consider the following cases:
\begin{description}
\item [Case 1] $\mu(E_i)>2g_i-2$ for $i=1,2$.
Since $E_i$ are semistable this implies that 
$h^1(E_i)=0$. So, by Riemann-Roch theorem,
$$k\leq h^0(E)= h^0(E_1)+h^0(E_2)-r=d_1+d_2+r(1-g_1-g_2)<d_1+d_2 +r$$
as $g_1+g_2 \geq 2.$
\item [Case 2] $\mu(E_i) > 2g_i-2$ and $\mu(E_j) \leq 2g_j -2$ for $i\neq j$.  We have $h^1(E_i) = 0$ so by Riemann-Roch we can computer $h^0(E_i)$. On the other hand, we can use Clifford's inequality in order to give a bound on $h^0(E_j)$. Then we get
$$k\leq h^0(E) =  h^0(E_i)+h^0(E_j)-r\leq d_i+r(1-g_i)+\frac{d_j}{2}+r-r< d_1 + d_2 + r$$
as $g_i\geq 2$ and $d_j>0$.
\item [Case 3] $\mu(E_i) \leq 2g_i-2$ for $i=1,2$. 
We use Clifford's inequality for a bound on both vector bundles:
$$ k\leq h^0(E) =  h^0(E_1) + h^0(E_2)-r \leq 
\frac{d_1}{2}+r+\frac{d_2}{2}+r-r<d_1+d_2+r.$$
Note that the last inequality holds since $d_i\geq 1$ for $i=1,2$. This follows since $E_i$ is semistable and $h^0(E_i(-p))\geq 1$ by assumption.
\end{description}
From the claim and inequality \ref{INEQ_E} we obtain $1 = w_1+w_2<1$ which is a contradiction.  
Hence $M_{E,V}$ is actually $w$-unstable for all polarizations $w$.
\end{proof}
Let $E$ be a globally generated  vector bundle of rank $r$ on the curve $C$. Then we can apply Lemma \ref{lemmatecnico1}, Remark \ref{REM:TObeprecise} and Theorem \ref{MEVunstable} to the generated pair $(E,H^0(E))$, in order to
obtain the following result on its kernel bundle $M_E$: 
 \begin{corollary}
 \label{cor1}
 Let $C$ be a reducible nodal curve as in Section \ref{sec1}. 
Let $E$ be a globally generated vector bundle on $C$.
If $E_1$ and $E_2$ are semistable and not trivial,
then the restriction of $M_E$ to each component $C_i$ is unstable and $M_E$ is strongly unstable. 
\end{corollary}

In particular, for line bundles we obtain the following: 

\begin{corollary}
\label{cor2}
Let $C$ be a reducible nodal curve as in Section \ref{sec1}. 
Let $L$ be a globally generated line bundle on $C$ with non trivial restrictions.
Then the restriction of $M_L$ to each component $C_i$ is unstable and $M_L$ is strongly unstable. 
\end{corollary}

\section{$w$-semistable kernel bundles}
\label{sec3}

Let $C$ be a nodal reducible curve with two smooth irreducible components $C_i$ of genus $g_i \geq 2$ and  a single node  $p$ as in Section \ref{sec1}.
Let $(E,V)$ be a generated pair on $C$, where $E$ is a vector bundle of rank $r \geq 1$ on $C$ and $V \subset H^0(E)$ of dimension $k \geq r+1$.  Let $E_i$ be the restriction of $E$ to the component $C_i$ and assume that $E_1$ and $E_2$ are both semistable. As we have seen in the previous section, when $V \cap H^0(E_j(-p)) \neq \{0\}$ for $j=1,2$, the kernel bundle of the pair $(E,V)$ is strongly unstable and both restrictions are unstable. In this section we would like to study generated pairs $(E,V)$ defining $w$-semistable kernel bundles with semistable restrictions. Hence it is natural to assume the following condition:
$$
(\star) \qquad \qquad \qquad \qquad V \cap H^0(E_1(-p))= V \cap H^0(E_2(-p)) = \{ 0 \}.\quad\qquad  \qquad \qquad \qquad \qquad$$
Let $M_{E,V}$ be the kernel bundle of the pair $(E,V)$ and $(E_i,V_i)$ the restriction of the pair $(E,V)$ to the component $C_i$. Then we have the following lemma:

\begin{lemma}
\label{lemmatecnico2}
Let $C$ be a reducible nodal curve as in Section \ref{sec1}.  Let $(E,V)$ be a generated pair, then  we have the following properties:
\begin{enumerate}
\item $(E,V)$ satifies $(\star)$ if and only if 
$M_{E,V} \otimes \OO_{C_i} \simeq M_{E_i,V_i}$, $i = 1,2$, 
where $M_{E_i,V_i}$ is the kernel bundle of $(E_i,V_i)$ on $C_i$;
\item under the above assumption, there exists a polarization $w = (w_1,w_2)$ such that 
$$ w_i \chi(M_{E,V}) \leq \chi(M_{E_i,V_i}) \leq w_i
\chi(M_{E,V}) + \Rk( M_{E,V}), \quad i= 1,2.$$
\end{enumerate}
\end{lemma}
\begin{proof}
(1) As we have seen in the proof of Lemma  \ref{lemmatecnico1},  the kernel of $\rho_i \colon H^0(E) \to H^0(E_i)$ is   $H^0(E_j(-p))$.   
So,  ${\rho_i}_{\vert V} \colon V \to V_i$ is  an isomorphism  if and only if 
$H^0(E_j(-p)) \cap V = \{0\}$. 
From  commutative diagramm \ref{mainCD} of Lemma \ref{lemmatecnico1},  $M_{E,V} \otimes \OO_{C_i} \simeq M_{E_i,V_i}$, if and only if  
$$S_j  = V \cap  H^0(E_j(-p))= \{ 0 \}.$$
(2) We want to prove the existence of   $w= (w_1,w_2) \in {\mathbb Q}^2$ with $0 < w_i < 1$ and $w_1 + w_2 = 1$ satisfying  the  conditions:
$$
w_i \chi(M_{E,V}) \leq 
\chi(M_{E_i,V_i})\leq w_i\chi(M_{E,V})+\Rk(M_{E,V}),
\quad i= 1,2. 
$$
Since $\chi(M_{E,V})= \chi(M_{E_1,V_1}) + \chi(M_{E_2,V_2}) - \Rk(M_{E,V})$, it is easy to verify that whenever  $w_1$ satisfies the above condition  for $i=1$, then $w_2=1-w_1$ satisfies the condition for $i = 2$.  So it is enough to find 
$w_1 \in (0,1) \cap {\mathbb Q}$ satisfying the following system of inequalities:
$$
\begin{cases}
\chi(M_{E_1,V_1}) \leq w_1 \chi(M_{E,V}) + \Rk(M_{E,V}) \\
\chi(M_{E_1,V_1}) \geq w_1 \chi(M_{E,V}).
\end{cases}
$$
From the exact sequences defining $M_{E,V}$ and $M_{E_1,V_1}$ one can get
$$\chi(M_{E,V})=(k-r)(1-p_a(C))-(d_1+d_2), \qquad \chi(M_{E_1,V_1})=(k-r)(1-g_1)-d_1$$
and $\Rk(M_{E,V}) = k-r$.
By substituting  in the above system  we obtain:
$$
\dfrac{(k-r)(g_1-1)+d_1}{(k-r)(p_a(C)-1)+(d_1+d_2)}\leq w_1\leq \dfrac{(k-r)g_1+d_1}{(k-r)(p_a(C)-1)+(d_1+d_2)}.$$
By denoting
$$ a_1 = \dfrac{(k-r)(g_1-1)+d_1}{(k-r)(p_a(C)-1)+(d_1+d_2)} \qquad b_1 = \dfrac{(k-r)g_1+d_1}{(k-r)(p_a(C)-1)+(d_1+d_2)}$$
it is immediate to see that 
$$0 < a_1 <  b_1  < 1$$
under the hypothesis of the lemma. Hence, for any $w_1 \in (a_1,b_1) \cap {\mathbb Q}$, we have that $(w_1,w_2)$  satisfies both the conditions of the lemma. 
\end{proof}

Using Lemma \ref{lemmatecnico2} we get the first result of this section.

\begin{theorem}
\label{THM:wstab}
Let $C$ be a nodal reducible curve as in Section \ref{sec1}. 
Let $(E,V)$ be a generated pair satisfying condition $(\star)$. If both $M_{E_1,V_1}$ and $M_{E_2,V_2}$ are semistable then there exists a polarization $w$ such that $M_{E,V}$ is $w$-semistable. Moreover, if $M_{E_i,V_i}$ is stable for at least one $i$, then $M_{E,V}$ is $w$-stable. 
\end{theorem}

\begin{proof}
We have to prove that $M_{E,V}$ is $w$-semistable for a suitable polarization. By Theorem \ref{stabilityconditions}, it is enough to verify that both  restrictions of $M_{E,V}$ to $C_1$ and $C_2$ are semistable and to find a polarization $w$ satisfying the    conditions \ref{COND:TEIXI}: 
\begin{equation}
    \label{COND}
 w_i \chi(M_{E,V}) \leq \chi(M_{E,V} \otimes \OO_{C_i}) \leq w_i
\chi(M_{E,V}) + \Rk( M_{E,V}), \quad i= 1,2.
\end{equation}
By Lemma \ref{lemmatecnico2} we have that the restrictions of $M_{E,V}$ are $M_{E_i,V_i}$,  which are semistable by assumption. Finally, conditions \ref{COND} are exactly the ones given in Lemma \ref{lemmatecnico2}.
The last assertion follows from Theorem \ref{stabilityconditions}.
\end{proof}

\begin{remark}
Let $C$ be  a nodal reducible curve as in Section \ref{sec1}.  
Let $E$ be a globally generated vector bundle of rank $r$ on the curve $C$.
Assume that there exist $k$ global sections generating $E$, with $r+1\leq k \leq \min(h^0(E_1),h^0(E_2))$. Then  a general pair $(E,V)$ with  $\dim V = k$ is generated and satisfies condition $(\star)$.
\end{remark}

\begin{proof}
For any $k \leq \min (h^0(E_1,h^0(E_2))$, let $G(k,H^0(E))$ denote the Grassmannian variety parametrizing  $k$-dimensional linear subspaces of $H^0(E)$.  Under our hypothesis, the set  
 of linear subspaces $V$ such that $(E,V)$ is generated is a non empty open subset.   So it is enough to show that a general pair $(E,V)$ satisfies $(\star)$.    Let's consider the Schubert cycle 
$$\sigma_j = \{ V \in G(k,H^0(E)) \,|\,  V \cap  H^0(E_j(-p)) \geq 1 \},$$
since  $k + h^0(E_j(-p)) \leq h^0(E)$, then it is a  proper closed subvariety of $G(k,H^0(E))$.  This concludes the proof.   
\end{proof}

The first application of this Theorem deals with the case of line bundles. For each smooth irreducible component $C_i$ of the nodal curve $C$, we denote by $\mathcal{G}^{k-1}_{d_i}(C_i)$ the variety parametrizing linear series of degree $d_i$ and dimension $k-1$ on the curve $C_i$. An element of  $\mathcal{G}^{k-1}_{d_i}(C_i)$ is given by a pair $(L_i,V_i)$ of type $(1,d_i,k)$ on $C_i$. We recall that, if $C_i$ is a general curve, then $\mathcal{G}^{k-1}_{d_i}(C_i)$ is non empty if and only if the Brill-Noether number $$\rho_i=g_i-k(g_i-d_i+k-1)$$ is non negative, hence if and only if $d_i \geq g_i + k-1 - \frac{g_i}{k}$. For details one can see \cite{ACGH}.

\begin{theorem}
\label{THM:MAIN2}
Let $C$ be a reducible nodal curve as in Section \ref{sec1}. Let $(L,V)$ be a generated pair on $C$, where $L$ is a line bundle and $\dim V = k$. Let $(L_i,V_i)$ be its restriction to $C_i$. If $C$ is general and $(L_i,V_i)$ is general in $\mathcal{G}^{k-1}_{d_i}(C_i)$, then there exists a polarization $w$ such that $M_{L,V}$ is $w$-semistable.
\end{theorem}

\begin{proof}  By Theorem \ref{THM:wstab} it is enough to show that $(L,V)$ satisfies condition $(\star)$ and that $M_{L_i,V_i}$ is semistable for $i=1,2$. 

Since $C$ is general the same holds for its components. Note that $\mathcal{G}^{k-1}_{d_i}(C_i)$ is non empty, 
since we are assuming $(L_i,V_i) \in \mathcal{G}^{k-1}_{d_i}(C_i)$.
Moreover, we  have that $$\dim(V_i)=k = \dim (V),$$
which implies that the restriction map $\rho_i$ induces an isomorphism of $V$ into $V_i$. In particular,
this gives us that $V\cap H^0(L_j(-p))=\{0\}$, i.e. condition $(\star)$ holds.
\hfill\par

It remains to show that, under the assumptions of the Theorem, the vector bundle $M_{L_i,V_i}$ is semistable. This follows from  \cite[Theorem 5.1]{BBN2} since $C_i$ is general and $(L_i,V_i)$ is general in $\mathcal{G}^{k-1}_{d_i}(C_i)$. 
\end{proof}

\begin{corollary}
Under the hypothesis of Theorem \ref{THM:MAIN2}, if we also  require  that a component $C_i$ is Petri-general, $k\geq 6$ and $g_i\geq 2k-6$, then $M_{L,V}$ is  $w$-stable.
\end{corollary}

\begin{proof}
By Theorem   \ref{THM:MAIN2} there exists a polarization $w$  such that  $M_{L,V}$ is $w$-semistable. According to Theorem \ref{stabilityconditions}, it is enough to show that  one of the restrictions of $M_{L,V}$ is stable.  Since $C_i$ is Petri, $k \geq 6$ and  $g_i \geq 2k-6$, then  $M_{L_i,V_i}$ is stable, from \cite[Theorem 6.1]{BBN2}, so we can conclude that $M_{L,V}$ is $w$-stable too.  
\end{proof}

The second application of Theorem \ref{THM:wstab} deals with generated pairs $(E,V)$ whose restrictions to each irreducible component $C_i$ is the complete  pair $(E_i,H^0(E_i))$. By Lemma \ref{lemmatecnico2}, as we are requiring condition $(\star)$ to hold, this occurs when $\rho_i|_V:V\to H^0(E_i)$ is an isomorphism.  In particular, this implies that in this case, $\dim(V) = h^0(E_1) = h^0(E_2)$. 

\begin{theorem}
\label{THM:3}
Let $C$ be a reducible  nodal curve as in Section \ref{sec1}. Let $E$ be a vector bundle on $C$ of rank $r \geq 1$. Assume that  its  restrictions $E_i$  are semistable of degree $d_i\geq 2rg_i$ and  satisfy  $d_1-d_2=r(g_1-g_2)$ so that $h^0(E_1)=h^0(E_2)=k$.
Then there exists a polarization $w$ such that for a general $V\in G(k,H^0(E))$ we have that $M_{E,V}$ is  $w$-semistable.  Moreover, if  $d_i>2rg_i$ for $i=1,2$, then $M_{E,V}$ is $w$-stable. 
\end{theorem}
\begin{proof}
Since $d_i\geq 2rg_i$ and $E_i$ is semistable we have that $E_i$ is globally generated, $h^1(E_i)=0$ and $h^0(E_i)=d_i+r(1-g_i)$. Using the same argument of the proof of Lemma \ref{lemmatecnico0}, we have that $h^0(E)=h^0(E_1)+h^0(E_2)-r$ (note that in the proof it is only needed that $E_i$ is globally generated). Hence we have $h^0(E)=2k-r$. \\

\noindent Now we will show that for general  $V\in G(k,H^0(E))$ we have that $\rho_i|_V$ is an isomorphism. Assume that $j \in\{1,2\}$ with $j\neq i$. As $E_j$ is globally generated we have $h^0(E_j(-p))=h^0(E_j)-r=k-r$. From the exact sequence 
$$
\xymatrix{
0 \ar[r] & H^0(E_j(-p))\ar[r] & H^0(E) \ar[r]^-{\rho_i} & H^0(E_i)
}
$$
we have $\dim(\rho_i(H^0(E)))=2k-r-(k-r)=k=h^0(E_i)$ so $\rho_i$ is surjective. 

\noindent Notice that $k+h^0(E_j(-p))=h^0(E)$, hence, for a  general subspace $V$ of $H^0(E)$ of dimension $k$ we have  $V\cap H^0(E_j(-p))=\{0\}$. This imply that $\rho_i|_V$ is injective and hence an isomorphism as $\dim(V)=h^0(E_i)$.  
Moreover, since $E_i$ is globally generated, for $V$ general in $G(k,H^0(E))$, $(E,V)$ is a generated pair and satisfies condition $(\star)$.
\\

By Theorem \ref{THM:wstab}, in order to conclude it is enough to prove that both $M_{E_i,V_i}$ are semistable. By hypothesis,
we have
$M_{E_i,V_i} = M_{E_i}.$
Hence, the result follows from \cite{B1}, Theorem 1.2: $M_{E_i}$ is semistable when $d_i \geq 2rg_i$  and it is stable if $d_i > 2rg_i$. 
\end{proof}


\begin{thebibliography}{[M-F]}

\bibitem{ACGH}
E. Arbarello, M. Cornalba, P. Griffiths, J. Harris,
{\em Geometry of algebraic curves. {V}ol. {I}} {\bf 267}, (1985),
Springer-Verlag, New York.

\bibitem{Be}
A.Beauville,
{Some stable vector bundles with reducible theta divisor},
{\it Manuscripta Math.} {\bf 110}, (2003), n 3,  343-349.

\bibitem{BBN1}
U.N.Bhosle, L. Brambila-Paz, P.E. Newstead,
{On coherent systems of type $(n,d,n+1)$ on Petri curves}, {\it Manuscripta Math.}  {\bf 126},(2008), n 4,  409-441.

\bibitem{BBN2}
U.N. Bhosle, L. Brambila-Paz, P.E. Newstead,
{ On linear series and a conjecture of D. C. Butler}, {\it Internat. J. Math.} {\bf 26} (2015), n. 2, 1550007, 18pp. 

\bibitem{BB}
M.Bolognesi, S.Brivio, 
{Coherent systems and modular subvarieties of $SU_C(r)$},
{\it Internat. J. Math.}  {\bf 23}, n 4, (2012), 1250037, 23pp. 

\bibitem{B-P}
L.Brambila-Paz, 
{ Non-emptyness of moduli spaces of coherent systems}, {\it Internat. J. Math.} {\bf 19}, (2008), n 7, 779-799. 

\bibitem{BGN}
L. Brambila-Paz, I. Grzegorczyk, P. Newstead,
{Geography of Brill-Noether loci for small slopes}, {\it J. Algebraic. Geom.} {\bf 6} (1997), n 4, 645-669.  

\bibitem{BGMN}
S.B.Bradlow, O. Garcia-Prada, V.Munoz, P.E. Newstead,
{Coherent systems and Brill-Noether theory}, {\it Internat. J. Math.} {\bf 14}, (2003), n 7, 683-733.

\bibitem{B15}
S.Brivio, 
{A note on theta divisors of stable bundles}, {\it Rev. Mat. Iberoam.} {\bf 31}, (2015), n 2, 601-608.

\bibitem{B17}
S.Brivio
{Families of vector bundles and linear systems of theta divisors}, {\it Internat. J. Math.}, {\bf 28}(2017), n 6, 1750039, 16pp. 

\bibitem{Br}
S.Brivio,
{Theta divisors and the geometry of tautological model}, {\it Collect. Math.}, {\bf 69},(2018), n 1,   131-150. 

\bibitem{BF1}
S. Brivio, F. Favale,
{ Genus 2 curves and generalized theta divisors}, 
{\it Bull. Sci. Math}, {\bf 155}, (2019), 112-140 

\bibitem{BF2}
S.Brivio, F. Favale,
{On vector bundles over reducible curves with a node},
To appear on {\it Adv. Geom.} (2019)

\bibitem{BV}
S.Brivio, A. Verra, 
{ Pluecker forms and the theta map}, {\it Am. J. Math.} {\bf 134}, (2012), n 5,  1247-1273.

\bibitem{B1} 
D.Butler, {Normal generation of vector bundles over a curve}, {\it J. Differential Geom.}, {\bf 39} (1994), n 1,  1--34.  

\bibitem{B2}
D.Butler, 
{ Birational maps of moduli of Brill-Noether pairs}, Preprint arXiv:alg-geom/9705009v1 (1997). 

\bibitem{C}
C. Camere, 
{About the stability of the tangent bundle restricted to a curve}, {\it C. R. Math. Acad. Sci. Paris}, {\bf 346} (2008), n 7-8, 421--426 

\bibitem{DM}
P.Deligne, D.Mumford, 
{The irreducibility of the space of curves of given genus},
{\it Inst. Hautes Etudes Sci. Publ. Math.},  {\bf 36} (1969), 75--120.

\bibitem{EL}
L.Ein, R. Lazarsfeld, 
{Stability and restrictions of Picard bundles, with an application to the normal bundles of elliptic curves}, {\it Complex projective geometry} (Trieste, 1989/Bergen 1989), London Math. Soc. Lecture Note Ser.  {\bf 179}, (1992), Cambridge Univ. Press, Cambridge, pp. 149-156. 

\bibitem{FT}
M. Franciosi, E. Tenni
{On {C}lifford's theorem for singular curves},
{\it Proc. Lond. Math. Soc.} (3),
{\bf 108}, (2014), n 3,  225--252

\bibitem{Gie}
D. Gieseker, {\em Lectures on moduli of curves}, Tata Institute of Fundamental Research, Lectures on Mathematics and Fhysics, {\bf 69}, (1982), Springer-Verlag. 

\bibitem{KN}
A.D. King, P. Newstead, { Moduli of {B}rill-{N}oether pairs on algebraic curves}, {\it Internat. J. Math.} {\bf 6}, (1995), n 5,  733-748.

\bibitem{La}
R. Lazarsfeld, {A sampling of vector bundles techniques in the study of linear series}, {\it Lectures on Riemann surfaces}, Trieste 1987,  500-559, World Sci. Pub. Teaneck NJ (1989). 

\bibitem{M}
E.C. Mistretta, 
{Stability of line bundle trasforms on curves with respect to two codimensional subspaces}, {\it J. Lond. Math. Soc.} (2), {\bf 78} (2008), no. 1, 172-182.

\bibitem{MS}
E.C. Mistretta, L. Stoppino,
{Linear series on curves: stability and Clifford index}, {\it Internat. J. Math.} {\bf 23},(2012), no 12, 1250121, 25 pp. 

\bibitem{Po}
M. Popa,
{Generalized theta linear series on moduli spaces of vector bundles on curves}, {\it Handbook of moduli}, {\bf III}, 219-255, {\it Adv. Lect. Math.} (ALM) {\bf 26}, (2013) Int. Press, Somerville, MA.

\bibitem{S}
 C.S.Seshadri, {\em Fibres vectorials sur les courbes algebriques}, Asterisque {\bf 96}, (1982).
 
 \bibitem{T1}
 M.Teixidor, { Moduli spaces of vector bundles on reducible curves}, {\it Amer. J. Math.}, {\bf 117}, (1995), n 1, 125-139.

\bibitem{T2} 
M.Teixidor, {Vector bundles on reducible curves and applications},
{\it Grassmannian, moduli spaces and vector bundles}, 169-180,  Clay Math. Proc.,{\bf 14}, Amer. Math. Soc., Providence, RI, (2011). 


\end{thebibliography}
\end{document}